\documentclass[11pt]{amsart}
\newtheorem{theorem}{Theorem}[section]

\usepackage{setspace}

\newtheorem{lemma}[theorem]{Lemma}
\newtheorem{proposition}[theorem]{Proposition}

\theoremstyle{definition}

\usepackage{mathabx}
\usepackage{amssymb}

\usepackage[paper=letterpaper, lmargin=1.309in,rmargin=1.309in,bottom=1.71353in,top=1.58688in,twoside=False,marginparwidth=0.811in,centering]{geometry}

\title{The $q$-Queens Problem: One-Move Riders on the Rectangular Board}
\author{Jaimal Ichharam}
\date{\today}
\address{Harvard University, Cambridge, MA 02138, USA.} 
\email{jichharam@college.harvard.edu}
\begin{document}

\begin{onehalfspace}

\begin{abstract}
We generalize the recent results of Chaiken et al. in \cite{qQueensII} to a rectangular $m\times n$ chessboard. An explicit formula
for the number of nonattacking configurations of one-move riders on such a chessboard is calculated in two different ways, one utilizing the theory of
symmetric functions and the other the theory of generating functions. With these newly found results, several conjectures and open problems in \cite{qQueensII}
are resolved, and various formulas found by Kotesovec in \cite{MathChess} are generalized.
\end{abstract}

\maketitle

\section{Introduction}
The famous $n$-Queens problem has had a long and storied history since its introduction in 1850. The problem, which consists of determining the number of ways to place $n$ nonattacking queens on an $n\times n$ chessboard, and its subsequent generalizations have
eluded the attempts of many researchers to find a closed form solution and inspired the development of entire branches of combinatorial mathematics,
including the theory of rook polynomials. Since the original puzzle was published, the scope of the problem has drastically expanded, with interest raised in studying the nonattacking configurations of various different types of pieces on arbitrarily shaped boards.\\

Recent work by Chaiken, Hanusa, and Zavlasky in \cite{qQueensI} has since framed these questions in an algebraic perspective by relating the nonattacking configurations of pieces on a chessboard to lattice polytopes and their associated Ehrhart quasipolynomials. In this work, we will combine some of their ideas along with various combinatorial techniques in order to further develop the results of \cite{qQueensII} on a rectangular chessboard, proving the natural generalizations of many of the theorems found in the work and establishing some of their unresolved conjectures. In particular, we obtain two different formulas for the number of nonattacking configurations of $q$ one-move riders on an $m\times n$ rectangular board, and prove
the conjecture found in \cite{qQueensII} that the quasipolynomial associated with the number of such configurations on a square board has $d$-periodic
coefficients.
\section{Two Pieces}
We begin by establishing the generalization of Lemma 3.1 in \cite{qQueensII} to a rectangular board. As in the aforementioned work, we consider a move $(c,d)\in\textbf{M}$ with slope $d/c$, and define the multiset of line sizes
\[\textbf{L}^{d/c}(m,n):=\{|l_{\mathcal{B}}^{d/c}(b)|:l_{\mathcal{B}}^{d/c}(b)\neq\emptyset\},\]
where
\[l_{\mathcal{B}}^{d/c}(b):=l^{d/c}(b)\cap [m,n].\]

\begin{lemma}
Assume $c$ and $d$ are relatively prime integers such that $0<\lfloor \frac{n}{d}\rfloor\leq\lfloor\frac{m}{c}\rfloor$. Let $\bar{n}:=n\pmod d$. Then the multiplicities of the line sizes of $\mathbf{L}^{d/c}(m,n)$ are given in the following table:
\[1\leq l<\lfloor\tfrac{n}{d}\rfloor: 2cd\hspace{1cm}\lfloor\tfrac{n}{d}\rfloor:(d-\bar{n})(m-c\lfloor\tfrac{n}{d}\rfloor)+c(\bar{n}+d)\hspace{1cm}\lfloor\tfrac{n}{d}\rfloor+1:\bar{n}(m-c\lfloor\tfrac{n}{d}\rfloor).\]
\end{lemma}
\begin{proof}
The analogous proof found in \cite{qQueensII} holds, with $m$ substituted for $n$ in the appropriate locations. The hypotheses of the lemma are modified from that of \cite{qQueensII} to ensure that the analysis is correct.
\end{proof}
We note that by exchanging $m$ and $c$ with $n$ and $d$ if necessary, we can always satisfy the hypotheses for Lemma 2.1 regardless of the given board and slope $d/c$. We can therefore also calculate the number of ordered $p$-tuples of collinear attacking positions on rectangular boards (Proposition 3.1 of \cite{qQueensII}). Define:
\[\alpha^{d/c}(p;m,n)=\hspace{-.4cm}\sum_{l\in\mathbf{L}^{d/c}(m,n)}\hspace{-.4cm}l^p.\]
and notice that under the hypotheses of Lemma 2.1:
\begin{align*}
\alpha^{d/c}(p;m,n)&=2cd\sum_{l=1}^{\lfloor\tfrac{n}{d}\rfloor-1}l^p+\left((d-\bar{n})(m-c\lfloor\tfrac{n}{d}\rfloor)+c(\bar{n}+d)\right)\lfloor\tfrac{n}{d}\rfloor^p+(m-c\lfloor\tfrac{n}{d}\rfloor)(\lfloor\tfrac{n}{d}\rfloor+1)^p,
\end{align*}
which yields the rectangular lattice generalizations of $\alpha^{d/c}$ and $\beta^{d/c}$ from \cite{qQueensII} for $p=2,3$:
\begin{align*}
\alpha^{d/c}(2;m,n)&=\left(\frac{3dmn^2-cn^3}{d^2}+\frac{c}{3}n\right)+\frac{\bar{n}(d-\bar{n})}{d^2}\left(dm-cn-\frac{c(d-2\bar{n})}{3}\right)\\
\alpha^{d/c}(3;m,n)&=\left(\frac{2dmn^3-cn^4}{2d^3}+\frac{c}{2d}n^2\right)\\
&+\frac{\bar{n}(d-\bar{n})}{d^3}\left((3n+d-2\bar{n})dm-(3n+2d-4\bar{n})cn+\frac{3c\bar{n}(d-\bar{n})}{2}\right).
\end{align*}
We can therefore calculate the number of nonattacking configurations for two pieces of the same type $\mathbb{P}$ on a rectangular board.
\begin{theorem}
For $(c,d)\in\mathbf{M}$, let $\hat{c}=\min(c,d)$ and $\hat{d}=\max(c,d)$, and denote by $m_{\hat{c},\hat{d}}$ and $n_{\hat{c},\hat{d}}$ a choice of orientation such that the hypotheses of Lemma 2.1 are satisfied, i.e. $0<\lfloor \frac{n}{\hat{d}}\rfloor\leq \lfloor\frac{m}{c}\rfloor$. Then we have:
\begin{align*}
u_\mathbb{P}(2;m,n)&=\frac{1}{2}m^2n^2-\frac{|\mathbf{M}-1|}{2}mn\\
&-\sum_{(c,d)\in\mathbf{M}}\left[\left(\frac{3\hat{d}m_{\hat{c},\hat{d}}n_{\hat{c},\hat{d}}^2-\hat{c}n_{\hat{c},\hat{d}}^3}{\hat{d}^2}+\frac{\hat{c}}{3}n_{\hat{c},\hat{d}}\right)+\frac{\bar{n}_{\hat{c},\hat{d}}(\hat{d}-\bar{n}_{\hat{c},\hat{d}})}{\hat{d}^2}\left(dm_{\hat{c},\hat{d}}-\hat{c}n_{\hat{c},\hat{d}}-\frac{\hat{c}(\hat{d}-2\bar{n}_{\hat{c},\hat{d}})}{3}\right)\right].
\end{align*}
\end{theorem}
\begin{proof}
We first calculate the number of attacking configurations, noting that this can only occur if the two pieces are along the same line of slope $(c,d)$ for some move in $\mathbf{M}$:
\[a_\mathbb{P}(2;m,n)=\hspace{-.2cm}\sum_{(c,d)\in\mathbf{M}}\hspace{-.2cm}\alpha^{d/c}(2;m_{\hat{c},\hat{d}},n_{\hat{c},\hat{d}})-(|\mathbf{M}|-1)mn.\]
Then we can calculate the number of attacking configurations:
\begin{align*}
u_\mathbb{P}(2;m,n)&=\frac{1}{2!}o_{\mathbb{P}}(2;m,n)=\frac{1}{2}\left[m^2n^2-\hspace{-.2cm}\sum_{(c,d)\in\mathbf{M}}\hspace{-.2cm}(2;m_{c,d},n_{c,d})+(|\mathbf{M}|-1)mn\right]\\
&=\frac{1}{2}m^2n^2-\frac{|\mathbf{M}-1|}{2}mn\\
&-\sum_{(c,d)\in\mathbf{M}}\left[\left(\frac{3\hat{d}m_{\hat{c},\hat{d}}n_{\hat{c},\hat{d}}^2-\hat{c}n_{\hat{c},\hat{d}}^3}{\hat{d}^2}+\frac{\hat{c}}{3}n_{\hat{c},\hat{d}}\right)+\frac{\bar{n}_{\hat{c},\hat{d}}(\hat{d}-\bar{n}_{\hat{c},\hat{d}})}{\hat{d}^2}\left(dm_{\hat{c},\hat{d}}-\hat{c}n_{\hat{c},\hat{d}}-\frac{\hat{c}(\hat{d}-2\bar{n}_{\hat{c},\hat{d}})}{3}\right)\right].
\end{align*}
\end{proof}
\section{One-Move Riders}
We proceed to generalize Proposition 6.1 from \cite{qQueensII}, calculating values of $u_\mathbb{P}(p;m,n)$ for $\mathbb{P}$ a one-move rider ($\mathbf{M}=\{(c,d)\}$ with $c,d$ relatively prime) on any convex board with line set $\mathbf{L}^{d/c}(\mathcal{B})$.

\begin{proposition}
For a piece $\mathbb{P}$ with moveset $\mathbf{M}=\{(c,d)\}$ with $0\leq c\leq d$ on a convex board $\mathcal{B}$,
\[u_\mathbb{P}(q;\mathcal{B})=\hspace{-.8cm}\sum_{n_1\lambda_1+\dots+n_k\lambda_k=q}\hspace{-.8cm}(-1)^{q-\sum_{i=1}^kn_i}\prod_{i=1}^k\frac{1}{\lambda_i^{n_i}n_i!}\left(\sum_{l\in\mathbf{L}^{d/c}}l^{\lambda_i}\right)^{n_i},\]
where the outermost sum represents a sum over all integer partitions of $q$, the $\lambda_i$ denote integers appearing in a particular partition and the $n_i$ their multiplicity.
\end{proposition}
\begin{proof}
We notice that in order for $q$ pieces on board $\mathcal{B}$ to be in a nonattacking configuration, it must be the case that the $q$ pieces are on $q$ distinct lines $l_1,\dots,l_q\in\mathbf{L}^{d/c}$. Therefore the number of nonattacking configurations is given by a sum over all possible combinations of $q$ lines:
\[u_\mathbb{P}(q;\mathcal{B})=\hspace{-.6cm}\sum_{\{l_1,\dots,l_q\}\subset\mathbf{L}^{d/c}}\hspace{-.4cm}l_1l_2\cdots l_q.\]
While the above equation does indeed yield a closed-form expression for the number of nonattacking configurations, enumerating subsets of a multiset is a difficult task. Instead, using the principle of inclusion-exclusion, we obtain a weighted alternating sum over partitions of $q$ consisting of products of sums of powers indexed by the multiset $\mathbf{L}$. An elementary combinatorial argument relates the coefficient of each summand to the order of the centralizer of the permutation group with $\sum_{i=1}^kn_i$ elements and cycle type $\lambda_1,\dots,\lambda_k$, which is calculated in \cite{SymmFunc}. The above expression then reads
\[u_\mathbb{P}(q;\mathcal{B})=\hspace{-.8cm}\sum_{n_1\lambda_1+\dots+n_k\lambda_k=q}\hspace{-.8cm}(-1)^{q-\sum_{i=1}^kn_i}\prod_{i=1}^k\frac{1}{n_i!}\left(\frac{1}{\lambda_i}\sum_{l\in\mathbf{L}^{d/c}}l^{\lambda_i}\right)^{n_i}.\]
\end{proof}
Restricting to the case of a $m\times n$ rectangular board, and orienting such that the hypotheses of Lemma 2.1 are satisfied, the sum inside the parentheses can be written:
\[\sum_{l\in\mathbf{L}^{d/c}(m,n)}\hspace{-.5cm}l^{\lambda_i}=2cd\sum_{l=1}^{\lfloor\frac{n}{d}\rfloor-1}l^{\lambda_i}+\left(d(c+m)-cn+(2c-m+c\lfloor\tfrac{n}{d}\rfloor)\bar{n}\right)\lfloor\tfrac{n}{d}\rfloor^{\lambda_i}+\bar{n}(m-c\lfloor\tfrac{n}{d}\rfloor)\left(\lfloor\tfrac{n}{d}\rfloor+1\right)^{\lambda_i},\]
by Lemma 2.1. Conjecture 6.1 from \cite{qQueensII}, which establishes the period of the coefficients of the quasipolynomial $u_\mathbb{P}(q;n,n)$ for a one-move rider on a square board, easily follows from the above:
\begin{theorem}
For a one-move rider with basic move $(c,d)$, the period of $u_\mathbb{P}(q;n,n)$ is exactly $\max(|c|,|d|)$.
\end{theorem}
\begin{proof}
It suffices to fix $d\geq c$. By the previous discussion, we have an explicit formula for the quasipolynomial:
\[u_\mathbb{P}(q;n,n)=\hspace{-.8cm}\sum_{n_1\lambda_1+\dots+n_k\lambda_k=q}\hspace{-.8cm}(-1)^{q-\sum_{i=1}^kn_i}\prod_{i=1}^k\frac{1}{\lambda_i^{n_i}n_i!}\left(\sum_{l\in\mathbf{L}^{d/c}}l^{\lambda_i}\right)^{n_i},\]
Fixing $q$ in the above expression, we obtain a polynomial expression for $u_\mathbb{P}(q;n,n)$ in terms of quasipolynomials $\sum_{l\in\mathbf{L}^{d/c}}l^{\lambda_i}$, each of which are $d$-periodic. It then immediately follows that $u_\mathbb{P}(q;n,n)$, as a polynomial of $d$-periodic quasipolynomials, is also $d$-periodic. 
\end{proof}
While the above expression is well-suited to the proof of the periodicity of the coefficients, we can obtain an alternate, more computationally efficient formula for $u_\mathbb{P}(q;m,n)$ by noting that for a board $\mathcal{B}$ with multiset of lines $\mathbf{L}^{d/c}$, the coefficient of $x^{|\mathbf{L}|-q}$ in the following polynomial
\[\prod_{l\in\mathbf{L}^{d/c}}(x+l),\]
yields the value of $u_\mathbb{P}(q;\mathcal{B})$. In the case of the rectangular board, the above polynomial takes a particularly simple form:
\[(x+\lfloor\tfrac{n}{d}\rfloor)^{(d-\bar{n})(m-c\lfloor\tfrac{n}{d}\rfloor)+c(\bar{n}+d)}(x+\lfloor\tfrac{n}{d}\rfloor+1)^{\bar{n}(m-c\lfloor\tfrac{n}{d}\rfloor)}\prod_{i=1}^{\lfloor\tfrac{n}{d}\rfloor-1}(x+i)^{2cd}.\]
Calculating the generating function for the coefficients of the above polynomial therefore yields an expression for $u_\mathbb{P}(q;m,n)$:
\begin{theorem}
Let $0<\lfloor \frac{n}{d}\rfloor\leq \lfloor\frac{m}{c}\rfloor$ with $\gcd(c,d)=1$. For a one move rider with moveset $\mathbf{M}=\{(c,d)\}$ on an $m\times n$ rectangular board, the value of $u_\mathbb{P}(q;m,n)$ is given by
\[\sum_{k+j+x_1+\dots+x_{2cd}=q}{(d-\bar{n})(m-c\lfloor\tfrac{n}{d}\rfloor)+c(\bar{n}+d)\choose j}\lfloor\tfrac{n}{d}\rfloor^j{\bar{n}(m-c\lfloor\tfrac{n}{d}\rfloor)\choose k}(\lfloor\tfrac{n}{d}\rfloor+1)^k\prod_{i=1}^{2cd}\left[{\lfloor\tfrac{n}{d}\rfloor\atop \lfloor\tfrac{n}{d}\rfloor-x_i}\right],\]
where the square brackets denote unsigned Stirling numbers of the first kind and $k,j,x_1,\dots x_{2cd}$ are nonnegative integers.
\end{theorem}
\begin{proof}
We recall the well-known result (\cite{EnumComb}, Corollary 2.4.2) that the $q$th coefficient of the polynomial $\prod_{i=1}^{n}(x+n)=S(n,q)$. Similarly, the $p$th coefficient of $(x+m)^s$ is given by ${s\choose p}m^p$. Then the above formula follows by summing over all ordered partitions of $q$ into the $2(cd\lfloor\tfrac{n}{d}\rfloor+1)$ of the above components that form the given polynomial and multiplying the corresponding binomial coefficients and Stirling numbers within each summand.
\end{proof}
This expression can be seen as an improvement over that in Proposition 3.1, as when the one-move rider is fixed (i.e. $c$ and $d$ are fixed), then the outermost sum over fixed-length partitions simplifies into a finite number of sums which is easier to evaluate. In particular, the above formula can be used to rederive certain classical expressions for particular one-move riders on a rectangular board. Consider, for example, the semirook (denoted $\mathbb{Q}^{10}$ as per the convention in \cite{qQueensIII}), a piece which moves only vertically. A simple combinatorial argument gives:
\[u_{\mathbb{Q}^{10}}(q,m,n)={m\choose q}n^q\]
To obtain this expression from the formula given in Theorem 3.3,  simply consider that $\bar{n}=2cd=0$, and therefore the only term that survives is:
\[u_{\mathbb{Q}^{10}}(q,m,n)={(d-\bar{n})(m-c\lfloor\tfrac{n}{d}\rfloor)+c(\bar{n}+d)\choose q}\lfloor\tfrac{n}{d}\rfloor^q={m\choose q}n^q\]
In a similar fashion, we can also derive the number of nonattacking configurations for a semibishop, denoted $\mathbb{Q}^{01}$.
In this case, one of the binomial coeffients vanishes, and the expression simplifies to
\[u_{\mathbb{Q}^{01}}(q;m,n)=\sum_{l=0}^{m-n+1}n^{m-n-l+1}{m-n+1\choose l}\sum_{j=1}^{m+n-q-l}\left[{n\atop j}\right]\left[n\atop m+n-q-j-l+1\right],\]
where we note that $m\geq n$. Setting $m=n$, the above formula simplifies to the expression provided in \cite{MathChess} for semibishops on the square $n\times n$ board.\\

\section{Conclusion and Future Work}
In Proposition 3.1 and Theorem 3.3, we have found the first closed-form formulas that entirely solve the $q$-Queens problem on a rectangular board for a specific class of chess pieces, namely, one-move riders. This is a significant breakthrough in the course of $q$-Queens history, and we hope that some of the techniques utilized will be extended to make further progress on this problem. In particular, we hope to extend our results to two-move riders using the theory of rook polynomials and generating functions, although this task appears to be much more difficult due to our inability to obtain closed expressions for rook polynomials for the necessary classes of skew Ferrers boards.

\end{onehalfspace}
\end{document}